\documentclass[12pt, dvipsnames]{article}

\usepackage[a4paper, margin=1in]{geometry}				

\usepackage[utf8]{inputenc}
\usepackage[T1]{fontenc}

\usepackage{lmodern}

\usepackage{amsmath}        % extensions for typesetting of math
\usepackage{amsfonts}       % math fonts
\usepackage{amsthm}         % theorems, definitions, etc.
\usepackage{amssymb}
\usepackage{bm}             % boldface symbols (\bm)
\usepackage{bbm}			% for blackboard bold symbols
\usepackage{graphicx}       % embedding of pictures
\usepackage{fancyvrb}       % improved verbatim environment
\usepackage[style=alphabetic,natbib=true,useprefix=true,sorting=nyt, maxcitenames=3, maxbibnames=10, mincitenames=1]{biblatex}
\usepackage{dcolumn}        % improved alignment of table columns
\usepackage{booktabs}       % improved horizontal lines in tables
\usepackage{paralist}       % improved enumerate and itemize
\usepackage{xcolor}  % typesetting in color

\usepackage{enumerate}
\usepackage{mathtools} % for DeclarePairedDelimiter

\usepackage[margin=1cm]{caption}

\usepackage[super]{nth}	%%% for nice typing of 1st, 2nd, etc..., just type \nth{1}, \nth{2}, ...

\usepackage{url}

\usepackage{color}
\definecolor{grey}{rgb}{.7,.7,.7}
\definecolor{blue}{rgb}{0,0,.8}
\definecolor{red}{rgb}{.8,0,0}
\definecolor{orange}{rgb}{1,.3,0}
\definecolor{green}{rgb}{0,.4,0}
\definecolor{gold}{rgb}{0.8,0.6,0.1}
\definecolor{brown}{rgb}{0.8,0.4,0.1}

\usepackage{hyperref}
\definecolor{blue3}{rgb}{.1,.0,.4}
\hypersetup{colorlinks=true, linkcolor=red, urlcolor=blue3, citecolor=green, pdfpagemode=UseNone, pdfstartview=FitH}
\hypersetup{unicode}
\hypersetup{breaklinks=true}
\hypersetup{hypertexnames=false} %without this the hyperreferences got messed up

\hypersetup{
	pdftitle    = {Three observations on the Colin de Verdière spectral graph parameter},
	pdfsubject  = {},
	pdfauthor   = {Vojtěch Kaluža, Vadym Koval},
	pdfkeywords = {Colin de Verdière parameter, spectrum of a graph, discrete Schrödinger operator, Strong Arnold property} ,
	pdfcreator  = {pdflatex},
	pdfproducer = {LaTeX with hyperref}
}

\theoremstyle{plain}
\newtheorem{thm}{Theorem}
\newtheorem{lemma}[thm]{Lemma}
\newtheorem{prop}[thm]{Proposition}
\newtheorem{quest}[thm]{Question}
\newtheorem*{quest*}{Question}

\newtheorem*{cor*}{Corollary}
\newtheorem{obs}[thm]{Observation}
\newtheorem{conj}[thm]{Conjecture}

\theoremstyle{plain}
\newtheorem{defn}[thm]{Definition}

\theoremstyle{remark}

\newtheorem{example}[thm]{Example}
\newtheorem*{example*}{Example}

\DefineVerbatimEnvironment{code}{Verbatim}{fontsize=\small, frame=single}

\newcommand{\R}{\mathbb{R}}
\newcommand{\N}{\mathbb{N}}
\newcommand{\Q}{\mathbb{Q}}
\newcommand{\Z}{\mathbb{Z}}

\DeclareMathOperator{\supp}{supp}
\DeclareMathOperator{\rank}{rank}
\DeclareMathOperator{\corank}{corank}

\newcommand{\mb}[1]{\mathbf{#1}}

\newcommand{\mbbm}[1]{\mathbbm{#1}}
\newcommand{\mc}[1]{\mathcal{#1}}

\newcommand{\abs}[1]{\left|#1\right|}
\newcommand{\set}[1]{\left\{#1\right\}}
\newcommand{\cl}[1]{\overline{#1}}
\DeclarePairedDelimiter{\floor}{\lfloor}{\rfloor}
\DeclarePairedDelimiter{\br}{(}{)}

\author{Vojtěch Kaluža\thanks{Vojtěch Kaluža was fully funded by the Austria Science Fund (FWF) [M 3100-N].}
	\\ \small{Institute of Science and Technology Austria}
	\and Vadym Koval
	\\ \small{École Polytechnique Fédérale de Lausanne}}

\title{Three observations on the Colin de Verdière spectral graph parameter}

\date{}
\begin{document}

\maketitle

\begin{abstract}
In this small note, we collect several observations pertaining to the famous spectral graph parameter $\mu$ introduced in 1990 by Y. Colin de Verdière. This parameter is defined as the maximum corank among certain matrices akin to weighted Laplacians; we call them CdV matrices.

First, we answer negatively a question mentioned in passing in the influential 1996 survey on $\mu$ by van der Holst, Lovász, and Schrijver concerning the Perron--Frobenious eigenvector of CdV matrices.

Second, by definition, CdV matrices posses certain transversality property. In some cases, this property is known to be satisfied automatically. We add one such case to the list.

Third, Y. Colin de Verdière conjectured an upper bound on $\mu(G)$ for graphs embeddable into a fixed closed surface. Following a recent computer-verified counterexample to a continuous version of the conjecture by Fortier Bourque, Gruda-Mediavilla, Petri, and Pineault, we also check using computer that the analogous example shows the failure of the conjectured upper bound on $\mu(G)$ for graphs embeddable into 10-torus as well as to several other larger surfaces.
\end{abstract}

\section{Introduction}
In 1990, \citet{CdV_orig, CdV_orig_en} introduced an algebraic graph parameter $\mu$, which was motivated by his study of the maximum multiplicity of the second smallest eigenvalue of Schrödinger operators on Riemannian surfaces. Already in \cite{CdV_orig}, he proved that the parameter $\mu$ is minor-monotone and that it can characterize planar graphs. Additionally, Colin de Verdi\`ere~\cite{CdV_orig} conjectured that $\chi(G)$, the (vertex) chromatic number of a graph $G$, is upper bounded by $\mu(G)+1$. This conjecture, still widely open, which interpolates between the four color theorem and the notorious Hadwiger's conjecture, together with many other intriguing properties that $\mu$ enjoys have generated a substantial interest in the parameter in the last 30 years.

Given a simple, unoriented graph $G=(V,E)$, we will associate to it a class of symmetric matrices $M\in\R^{V\times V}$ such that 
\begin{enumerate}[(i)]
	\item $M_{uv}<0$ whenever $uv\in E$ and $M_{uv}=0$ if $uv\notin E, u\neq v$,
	\item $M$ has exactly one negative eigenvalue.
\end{enumerate}
There is no condition on the diagonal of $M$. The set of all such matrices associated to $G$ is denoted by $\mc{M}(G)$. These matrices are sometimes called \emph{discrete Schrödinger operators} in the literature, since they are analogous to (continuous) Schrödinger operators on differentiable manifolds.

The value of $\mu(G)$ is defined as the maximum corank among all matrices $M\in\mc{M}(G)$
that satisfy an additional transversality condition due to \citet{Arnold}:
\begin{defn}[Strong Arnold Property (SAP)]\label{def:SAP}
	We say that $M\in\mc{M}(G)$ satisfies the \emph{Strong Arnold Property (SAP)} if and only if there is no non-zero symmetric matrix $X\in\R^{V\times V}$ such that $MX=\mathbf{0}$ and
	\begin{equation*}
		X_{uv}=0 \qquad\text{ whenever } uv \in E \text{ or } u=v.
	\end{equation*}
\end{defn}
Note that $X$ in the definition above is forced to be zero everywhere where $M\in\mc{M}(G)$ is or can be nonzero. We call matrices $M\in\mc{M}(G)$ that satisfy (SAP) \emph{CdV matrices for~$G$}.

In this short note we discuss three questions regarding $\mu$. The first one concerns the definition of $\mu$. It is well-known (see \citep[Thm.~2.5]{CdV_main}) that when $G$ contains at least one edge, then $\mu(G)$ is equal to the maximum of $\mu$ over the connected components of $G$. Thus, in the study of $\mu$, we can restrict ourselves to connected graphs. When $G$ is connected, the Perron--Frobenius theorem applies to $M\in\mc{M}(G)$; it implies that the multiplicity of the smallest eigenvalue $\lambda_1(M)$ of $M$ is one, and moreover, the corresponding eigenvector, called the Perron--Frobenius eigenvector, is strictly positive (or strictly negative). \Citet[p.~22]{CdV_main} asked in passing (in a different, but equivalent form) the following question:
\begin{quest}[{\citep[p.~22]{CdV_main}}]\label{q:all1}
	Is it true that for every connected graph $G$ there exists a CdV matrix of corank $\mu(G)$ such that its Perron--Frobenius eigenvector is $\mbbm{1}$?
\end{quest}
We provide a counterexample to this question in Section~\ref{s:all_one}.

The second question discussed in the present paper concerns (SAP). It is a crucial ingredient in any proof of the minor-monotonicity of $\mu$ known so far \citep{CdV_orig, vdHolst_PhD, CdV_main}. It has several equivalent reformulations with very distinct flavors \citep[Prop.~1]{SS_flat_and_SAP}, \citep[Lem.~10.25, Ex.~10.10]{Lovasz_book}. Despite this, the property and its impact on $\mu(G)$ for a given graph $G$ is far from being understood. It is known that whether $M\in\mc{M}(G)$ satisfies (SAP) or not depends only on $\ker(M)$ and $G$ \citep[Ex.~10.10]{Lovasz_book}. Few consequences of (SAP) in terms of the connectivity of subgraphs of $G$ induced by vectors in $\ker(M)$ are known \citep[Sec.~2.5]{CdV_main}, but the exact relationship between (SAP) and the combinatorics of $\ker(M)$ remains mysterious.

There are several classes of graphs $G$ for which it is known that all $M\in\mc{M}(G)$ satisfy (SAP); namely, this is true whenever $G$ is a path, 2-connected outer planar, 3-connected planar \cite{vdHolst_PhD, vdHolst_SAP_and_connectivity} or 4-connected linkless embeddable graph \citep{SS_flat_and_SAP}. It is an interesting open question whether this list can be continued:
\begin{quest}[{\citep{SS_flat_and_SAP}}]\label{q:SAP_and_connectivity}
Given a $\mu(G)$-connected graph $G$, does every $M\in\mc{M}(G)$ satisfies (SAP)?
\end{quest}
We do not answer this question, but it motivated the proposition presented below. 

For general graphs $G$, it is easy to see that every $M\in\mc{M}(G)$ with $\corank(M)\leq 1$ has (SAP). In this direction, we add the following observation, proven in Section~\ref{s:SAP}:
\begin{prop}\label{prop:corank_2_SAP}
	Let $G=(V,E)$ be a connected graph and $M\in\mc{M}(G)$ a matrix of corank $2$. Then $M$ satisfies SAP.
\end{prop}

In analogy to the genesis of Question~\ref{q:SAP_and_connectivity}, one could wonder whether for every $k$-connected graph $G$ every $M\in\mc{M}(G)$ with $\corank(M)\leq k+1$ satisfied (SAP) or not. We do not know if it is true for $k=2$, but we show in Observation~\ref{obs:quest_on_SAP_tight} that it fails for $k\geq 3$. Even more strongly, we show there that for every $k\geq 3$ there is a $k$-connected graph $G$ and $M\in\mc{M}(G)$ with $\corank(M)=4$ that does not satisfy (SAP).

In the third part we present a computer-verified counterexample to the following conjecture:
\begin{conj}[{Colin de Verdière\footnote{The conjecture is apparently older than \citep{YCdV_survey_on_discrete_Schroedinger} demonstrates, as is witnessed by \citet[Conj.~16]{Pendavingh_PhD}, who attributed it to \citeauthor{YCdV_survey_on_discrete_Schroedinger}. However, the reference in \citep[Conj.~16]{Pendavingh_PhD} does not mention the conjecture. We are not aware of any explicit mention of the conjecture by \citeauthor{YCdV_survey_on_discrete_Schroedinger} older than \citep{YCdV_survey_on_discrete_Schroedinger}. On the other hand, it follows directly from Conjecture~\ref{c:CdV_surfaces_continuous} by Theorem~\ref{thm:discrete_to_cont}.}~\citep{YCdV_survey_on_discrete_Schroedinger}}]\label{c:CdV_surfaces}
	For every closed surface $S$ and every graph $G$ embedded into $S$ it holds that $\mu(G)\leq\gamma(S)-1$, where $\gamma(S)$ is the maximum of the vertex chromatic number $\chi(H)$ among all graphs $H$ embeddable into $S$.
\end{conj}
Since the resolution of Heawood's conjecture \citep{Ringel_Youngs-Heawood_orientable, Youngs-Heawood_nonorientable} it is known that $\gamma(S)$ is equal to the largest $n\in\N$ such that $K_n$ embeds into $S$ (the case of the sphere is special and is subject to the famous four color theorem). This, in turn, satisfies
\begin{equation}\label{eq:Heawood}
	\gamma(S)=\floor*{\frac{7+\sqrt{49-24\chi(S)}}{2}},
\end{equation}
where $\chi(S)$ denotes \emph{the Euler characteristic of $S$}, except when $S$ is the Klein bottle, in which case $\gamma(S)$ is one less than what \eqref{eq:Heawood} predicts.

Conjecture~\ref{c:CdV_surfaces} is a formally weaker version of the following conjecture made by Colin de Verdi\`ere:
\begin{conj}[{\citet{CdV_multiplicity_bounds_mu}, \citep[Ch.~4]{CdV_Spectre_Graphes}}]\label{c:CdV_surfaces_continuous}
	For every surface $S$, it holds that $m(S)=\gamma(S)-1$, where $m(S)$ stands for the maximum multiplicity of the second smallest eigenvalue among all (continuous) Schrödinger operators\footnote{We will not need any notions from Riemannian geometry nor the theory of partial differential operators formally in the present work, so we also do not define any of them here. We just use them to describe the origin of the problem considered here and what is known about it so far. An interested reader is referred to the work of Colin de Verdi\`ere cited in this note and the references therein.} on $S$.
\end{conj}
The relation between Conjectures~\ref{c:CdV_surfaces} and \ref{c:CdV_surfaces_continuous} is made clear by Theorem~\ref{thm:discrete_to_cont} below.	 
\begin{thm}[{\citet[Thm.~7.1]{CdV_multiplicity_bounds_mu}}]\label{thm:discrete_to_cont}
	For every surface $S$ and every graph $G$ embeddable into $S$ it holds that $\mu(G)\leq m(S)$.
\end{thm}
We note that $\mu(K_n)=n-1$ for every $n\in\N$ (see \citep{CdV_orig, CdV_orig_en}), and thus, $m(S)\geq \gamma(S)-1$ for every surface $S$. Hence, the difficult part of Conjecture~\ref{c:CdV_surfaces_continuous} is the upper bound.

When the surface $S$ is the sphere $S^2$, the projective plane $\R P^2$, the torus $T_1$ or the Klein bottle, the validity of Conjecture~\ref{c:CdV_surfaces_continuous} was established in \citep{Cheng_Schroedinger_sphere, Besson_Schroedinger_torus, CdV_multiplicity_bounds_mu}. For Conjecture~\ref{c:CdV_surfaces}, the cases of $S^2, \R P^2$ and $T_1$ were established directly in \citep{vdHolst_planar, Pendavingh_PhD}.

Currently the best general upper bound on $m(S)$ is $5-\chi(S)$, as was proven by \citet{Sevennec_general_Schroedinger_bound_surfaces}, and applies to all surfaces with $\chi(S)<0$. This extends the validity of Conjectures~\ref{c:CdV_surfaces_continuous} and \ref{c:CdV_surfaces} to all surfaces $S$ with $\chi(S)\geq -3$. No better bound is known in the discrete setting at present; the only general result in this direction that we are aware of and that does not use Theorem~\ref{thm:discrete_to_cont} was proven in \citep{Pendavingh_PhD, CdV_on_surface_lin_discrete} and shows that $\mu(G)\leq 7-2\chi(S)$ when $G$ is embeddable into $S$. Note that the conjectured upper bound is $O\br*{\sqrt{\abs{\chi(S)}}}$, while all bounds proven so far are of order $O(\abs{\chi(S)})$.

There has been quite a lot of research into Conjecture~\ref{c:CdV_surfaces_continuous} when restricted to hyperbolic surfaces and to Riemannian Laplacians instead of the general Schrödinger operators; while some of these works achieve sublinear dependence on $\chi(S)$, none of the results available comes close to the conjectured $O\br*{\sqrt{\abs{\chi(S)}}}$ bound. As these results are out of the scope of this note, we only refer to one of the most recent papers \cite{Surfaces_recent_survey} of this type, which also contains a nice summary of the present state.

In December 2023, \citet{CdV_surfaces_cont_counter-examples} identified two counterexamples to Conjecture~\ref{c:CdV_surfaces_continuous}; namely, using rigorous\footnote{By a rigorous computation we mean a computation output of which is not affected by numerical imprecision. At least in principle, such a computation could be used to generate a formal proof.} computation with \verb|SageMath|, they showed that $m(T_{10})\geq 16$ and $m(T_{17})\geq 21$, where $T_n$ denotes a surface homeomorphic to the sphere with $n$ handles attached to it. By \eqref{eq:Heawood}, the conjectured bounds in these cases are $\gamma(T_{10})-1=13$ and $\gamma(T_{17})-1=16$, respectively. The two examples are (continuous) Riemannian Laplacians of two \emph{triangle surfaces}, \texttt{T10.1} and \texttt{T17.1}, from the list of all triangle surfaces of genera 2 to 101 produced by \citet{Conder_list}.

Since Theorem~\ref{thm:discrete_to_cont} provides only one implication,
the aforementioned counterexamples do not automatically imply failure of Conjecture~\ref{c:CdV_surfaces}. We bridge this gap in Proposition~\ref{prop:counter_ex} and demonstrate the failure of Conjecture~\ref{c:CdV_surfaces} in the case of $T_{10}$ using rigorous computer verification in \verb|SageMath|~\citep{sagemath} in Section~\ref{s:surfaces}. As we discuss there, it actually implies failure of Conjecture~\ref{c:CdV_surfaces} for all surfaces $S$ with $\chi(S)\in [-28, -18)$.

All the code used to construct the counterexample and to verify its properties is included in the ancillary file \verb|code.sage|. 
\section{Preliminaries}
Here we summarize the notation and basic background used in the paper. The background specific to a particular section is stated at the beginning of the section.

\paragraph{Graphs.} For general background on graphs we refer to \citep{Diestel}. The graphs considered here are always finite, simple and undirected. Given a graph $G$, we denote by $V(G)$ and $E(G)$ its sets of vertices and edges, respectively, unless they are specified explicitly.
Whenever $u,v\in V(G)$, we write just $uv$ instead of $\set{u, v}$ for simplicity. If $S\subseteq V(G)$, we let $G[S]$ stand for the subgraph of $G$ induced by $S$, i.e., the graph $(S, \set{uv\in E(G)\colon u,v\in S})$. By a \emph{connectivity} of a graph we mean its vertex connectivity.

By $K_n$ we denote the complete graph on $n$ vertices and $K_{a,b}$ stands for the complete bipartite graph with parts of size $a$ and $b$.

\paragraph{Vectors and matrices.} For general background on matrices we refer to \citep{Horn--Johnson}. The zero vector and the zero matrix of appropriate dimension are denoted by $\mb{o}$ and $\mb{0}$, respectively. The inequalities like $\leq, >$ applied to vectors and matrices should be interpreted as applying to every entry.

Given a graph $G$ and $x\in\R^{V(G)}$, we write $\supp(x):=\set{u\in V(G)\colon x_u\neq 0}$ for the \emph{support of $x$}. Given a matrix $M\in\R^{n\times m}$, $\corank(M)$ refers to the dimension of $\ker(M)$.

We denote by $A_G\in\Z^{V(G)\times V(G)}$ the \emph{adjacency matrix} of the graph $G$, i.e., the symmetric $0/1$-matrix such that $(A_G)_{u,v}=1$ if and only if $uv\in E(G)$. By eigenvalues and a characteristic polynomial of $G$ we mean the eigenvalues and the characteristic polynomial of $A_G$.
Additionally, we write $I_G$ for the $V(G)\times V(G)$ identity matrix. For $S\subset V(G)$ we let $\mbbm{1}_S$ be the vector that is equal to $1$ on every $v\in S$ and 0 otherwise.

All matrices considered in this paper are real and almost all are symmetric. By the spectral theorem, for every $n\times n$ symmetric real matrix $M$ there is an orthogonal basis of $\R^n$ consisting of eigenvectors of $M$. Thus, $M$ can be diagonalized by its eigenvectors. Yet differently, $M$ can be written as a linear combination of rank-1 matrices generated by the eigenvectors of $M$.

\paragraph{Discrete Schrödinger operators.} Given a connected graph $G$ and $M\in\mc{M}(G)$ (defined in the introduction), there is $c>0$ such that $cI_G-M\geq\mb{0}$. Thus, the latter matrix is \emph{irreducible}, and hence, the classical Perron--Frobenius theorem applies to it ensuring that its largest eigenvalue is simple and the corresponding eigenvector is strictly positive (or strictly negative). Since adding $cI_G$ to a matrix shifts all its eigenvalues by $c$ and preserves all eigenvectors, we can say that the Perron--Frobenius theorem applies also to $M$, by which we mean that it ensure that the \emph{smallest} eigenvalue of $M$ is simple and that the corresponding eigenvector is strictly positive; we refer to these objects as the Perron--Frobenius eigenvalue and the Perron--Frobenius eigenvector of $M$.

As eigenvectors belonging to different eigenvalues are always orthogonal, the previous discussion also implies that whenever $z>\mb{o}$ is an eigenvector of $M\in\mc{M}(G)$ for $G$ connected, then $z$ is the Perron--Frobenius eigenvector of $M$.
\section{The Perron--Frobenius eigenvector of CdV matrices}\label{s:all_one}
Below we describe a counterexample to Question~\ref{q:all1}.
\begin{example}
    	Consider a connected graph $G$ on 7 vertices such that its complement is a disjoint union of a triangle and a $3$-star. By \citep[Thm.~5.4]{CdV_main}, such graph satisfies $\mu(G)=7-2=5$. By definition, its CdV matrix has the form 
    	\[
    	\begin{pmatrix}
    		? & 0 & 0 & * & * & * & *\\
    		0 & ? & 0 & * & * & * & *\\
    		0 & 0 & ? & * & * & * & *\\
    		* & * & * & ? & * & * & 0\\
    		* & * & * & * & ? & * & 0\\
    		* & * & * & * & * & ? & 0\\
    		* & * & * & 0 & 0 & 0 & ?
    	\end{pmatrix},
    	\]
    where each `$*$' symbol is to be replaced with a negative number, while each '?' can be any real number. In particular, there is no zero row/column.
    
    We know that an optimal CdV matrix has row rank $2$. To get a basis of the row space, we can take the \nth{1} row as one of the basis vectors. The second one cannot be the \nth{2} or the \nth{3} row, as each of them shares a common zero with the \nth{1} row and there is no all $0$ column vector in any CdV matrix for $G$. We can take the \nth{4} row as the second basis row, since it is always linearly independent from the \nth{1} row.
    By the above considerations, each CdV matrix of $G$ looks like the following: 
\[  
    \begin{pmatrix}
    	0 & 0 & 0 & * & * & * & *\\
    	0 & 0 & 0 & * & * & * & *\\
    	0 & 0 & 0 & * & * & * & *\\
    	* & * & * & ? & * & * & 0\\
    	* & * & * & * & ? & * & 0\\
    	* & * & * & * & * & ? & 0\\
    	* & * & * & 0 & 0 & 0 & ?
    \end{pmatrix}.
\]

	From now on we assume that $\mbbm{1}$ is an eigenvector of the matrix we consider. Since the graph is connected, the Perron--Frobenius theorem implies that $\mbbm{1}$ belongs to the smallest eigenvalue, which is negative by definition. Therefore, each row sums to the same negative number.
	
	We can see that the \nth{2} and the \nth{3} rows are just some multiples of the \nth{1}. From that we conclude that those three rows are identical. Thus, we derived that any CdV matrix of $G$ has the following form:
\[
	\begin{pmatrix}
	 	0 & 0 & 0 & a_1 & a_2 & a_3 & a_4\\
	 	0 & 0 & 0 & a_1 & a_2 & a_3 & a_4\\
	 	0 & 0 & 0 & a_1 & a_2 & a_3 & a_4\\
	 	a_1 & a_1 & a_1 & b & * & * & 0\\
	 	a_2 & a_2 & a_2 & * & ? & * & 0\\
	 	a_3 & a_3 & a_3 & * & * & ? & 0\\
	 	a_4 & a_4 & a_4 & 0 & 0 & 0 & ?
	 \end{pmatrix}.
\]
 
 	As \nth{5} and \nth{6} rows have zero last coordinate, they are some multiple of the \nth{4} row (as the \nth{1} row is non-zero in the last coordinate). Moreover, as they have the same sum of entries as the \nth{4} row, they all are identical. The same applies to columns. Therefore, we deduce that our CdV matrix looks as follows:
\[ 	
 	\begin{pmatrix}
 		0 & 0 & 0 & a & a & a & a_4\\
 		0 & 0 & 0 & a & a & a & a_4\\
 		0 & 0 & 0 & a & a & a & a_4\\
 		a & a & a & b & b & b & 0\\
 		a & a & a & b & b & b & 0\\
 		a & a & a & b & b & b & 0\\
 		a_4 & a_4 & a_4 & 0 & 0 & 0 & ?
 	\end{pmatrix}.
\]
 
 	The last coordinate of the \nth{7} row is uniquely determined by the requirement that this row is linearly dependent with the \nth{1} and the \nth{4} rows. Additionally, the sum of entries of the \nth{3} and the \nth{4} rows must match:
\[
	\begin{pmatrix}
 		0 & 0 & 0 & a & a & a & a_4\\
 		0 & 0 & 0 & a & a & a & a_4\\
 		0 & 0 & 0 & a & a & a & a_4\\
 		a & a & a & \frac{a_4}{3} & \frac{a_4}{3} & \frac{a_4}{3} & 0\\
 		a & a & a & \frac{a_4}{3} & \frac{a_4}{3} & \frac{a_4}{3} & 0\\
 		a & a & a & \frac{a_4}{3} & \frac{a_4}{3} & \frac{a_4}{3} & 0\\
 		a_4 & a_4 & a_4 & 0 & 0 & 0 & -\frac{a_4}{3} \frac{a_4^2}{a^2}
 	\end{pmatrix}.
\]
 
 	As sums of the entries of all the rows must be the same, we get $3a+a_4=3a_4-\frac{a_4^3}{3a^2}$. Since both $a$ and $a_4$ are negative, we can write $a_4=xa$ for some $x>0$. Thus, the last equation becomes $-\frac{x^3}{3}+2x-3=0$, which in turn is equivalent to $(x+3)(x^2-3x+3)=0$. However, this equation does not have any positive roots, completing the counterexample.
\end{example}
\section{The Strong Arnold Property and connected graphs}\label{s:SAP}
Given a graph $G=(V, E)$ and a symmetric real matrix $X\in\R^{V\times V}$ we say that $X$ is a \emph{$\cl{G}$-matrix} if $X_{uv}=0$ whenever $uv\in E(G)$.

We will use several facts. The first one is taken from \citet[Lem.~3]{Pendavingh_separation} (the statement is strengthened, however, the proof is the same.)
\begin{lemma}\label{l:components}
Let $G$ be a connected graph and $M\in\mc{M}(G)$. For every $x\in\ker(M)$, if $\supp(x)$ is disconnected, then
\begin{enumerate}
	\item\label{l:components_signs} each of its connected component is either a connected component of $\supp_+(x)$, or of $\supp_-(x)$;
	\item\label{l:components_number} it has exactly $\dim(D)+1$ connected components, where \[D:=\set{y\in\ker(M)\,\middle|\,\supp(y)\subseteq\supp(x)}.\]
\end{enumerate}
\end{lemma}

The second fact is a consequence of the spectral decomposition of symmetric real matrices.
\begin{lemma}\label{l:SAP_sum_rank1}
	Let $G$ be a graph and $M\in\mc{M}(G)$. Let $X$ be a $\cl{G}$-matrix with zero diagonal such that $MX=\mb{0}$. Then there are $x_1, \ldots, x_k\in\ker(M)$ such that $X=\sum_{i=1}^k \sigma_i\cdot x_i x_i^T$, where $\sigma_i\in\set{-1, 1}$ for every $i\in[k]$ and $k=\corank(M)$. If $X\neq \mb{0}$, not all $\sigma_i$'s can have the same sign.
\end{lemma}
\begin{proof}
By the spectral decomposition theorem, there are mutually orthogonal vectors $x_1, \ldots, x_k\in\R^{V(G)}$ such that $k=\rank(X)$ and $X=\sum_{i=1}^k \sigma_i\cdot x_i x_i^T$, where $\sigma_i\in\set{-1, 1}$ for every $i\in[k]$. Since $MX=\mb{0}$, the rank-nullity theorem implies that $k\leq\corank(M)$. We can assume that $k=\corank(M)$ by adding $\mb{o}$ an appropriate number of times to the collection $x_1, \ldots, x_k$ if needed.

By the spectral decomposition theorem again, there is an orthogonal basis of eigenvectors of $M$. Those belonging to the nonzero eigenvalues of $M$ form a basis of the row space of $M$; let $z$ be such an eigenvector. Then $MX=\mb{0}$ implies that $z$ is orthogonal to every column of $X$. Since the columns of $X$ are generated by the orthogonal set of vectors $x_1, \ldots, x_k$, we conclude that $z^T x_i=0$ for every $i\in[k]$. Thus, each $x_i$ is orthogonal to the whole row space of $M$, and hence, $x_i\in\ker(M)$ for every $i\in[k]$.

For every $v\in V(G)$, as $0=X_{vv}=\sum_{i=1}^k\sigma_i(x_{i})_v^2$ by the assumption, if all $\sigma_i$ have the same sign, then $(x_i)_v=0$ for all $i\in [k]$; thus, $x_i=\mb{o}$ for every $i\in[k]$, implying that $X=\mb{0}$.
\end{proof}

Now we are ready to prove Proposition~\ref{prop:corank_2_SAP}.
\begin{proof}[Proof of Proposition~\ref{prop:corank_2_SAP}]
For contradiction, assume $MX=0$ for some non-zero $\cl{G}$-matrix with zero diagonal. By Lemma~\ref{l:SAP_sum_rank1}, we know that there are $x,y\in\ker(M)$ such that $X=xx^T-yy^T$. As $X$ has zero diagonal, we see that $x_{v}^2=y_{v}^2$ for every $v\in V(G)$. Thus, $\supp(x)=\supp(y)$. If $x=\pm y$, then $xx^T=yy^T$ and $X=\mb{0}$, which is a contradiction.

Therefore, we can assume that there are $v,w\in V(G)$ such that $x_v=y_v$ and $x_w=-y_w$. As $X$ is $\overline{G}$-matrix, it holds that $x_vx_w=y_vy_w$ whenever $vw\in E(G)$. Therefore, $G[\supp(x)]$ cannot be connected.

Let $C_1, \ldots, C_k$ stand for the sets of vertices of the connected components of $G[\supp(x)]$. By Lemma~\ref{l:components}\eqref{l:components_signs}, both $x$ and $y$ have a constant sign on each $C_i$, $i\in[k]$. We can think of $M$ as a block matrix, where the blocks are indexed by $C_1, \ldots, C_k$ and $V\setminus\supp(x)$. As $M\in\mc{M}(G)$, $M_{C_i\times C_j}\equiv\mb{0}$ for every $i\neq j, i,j\in[k]$. Consequently, for every $i\in[k]$, the fact that $x,y\in\ker(M)$ implies that $M_{C_i} x_{C_i}=M_{C_i} y_{C_i}=\mb{o}$. By the Perron--Frobenius theorem, the smallest eigenvalue of $M_{C_i}$ is thus zero, has multiplicity one, and additionally, $x_{C_i}>\mb{o}$ or $x_{C_i}<\mb{o}$ and the same is true for $y$. Thus, $x_{C_i}=\alpha_{C_i} y_{C_i}$, where $\alpha_{C_i}\in\set{-1, 1}$ for every $i\in[k]$. As $x\neq\pm y$, we see that not all $\alpha_{C_i}$ have the same sign.

Let $\pi$ be the Perron--Frobenius eigenvector of $M$. Then
\begin{equation}\label{eq:signs}
\pi^T x=\sum_{i=1}^k\pi_{C_i}^T x_{C_i}=\sum_{i=1}^k\pi_{C_i}^T y_{C_i}=\pi^T y=0.
\end{equation}
Each of the terms $\pi_{C_i}^T x_{C_i}$ is non-zero, but their sum is zero. This means that they cannot all have the same sign. The same argument applies also to $y$. For $k=2$ we thus get that $x=\pm y$, which is a contradiction.

If $k=3$, then precisely two of $\alpha_{C_i}, i\in[3]$, have the same sign, say, $\alpha_{C_1}$ and $\alpha_{C_2}$. Using that $x_{C_i}=\alpha_{C_i}y_{C_i}$ for every $i\in[k]$, \eqref{eq:signs} then shows that also $\alpha_{C_3}$ must have the same sign as $\alpha_{C_1}=\alpha_{C_2}$, as each of the terms $\pi_{C_i}^T x_{C_i}$ is nonzero and they sum to zero.

Consequently, $k\geq 4$. However, Lemma~\ref{l:components}\eqref{l:components_number} then identifies a subspace of $\ker(M)$ of dimension at least three, which is the final contradiction.
\end{proof}

We would like to discuss some limits on possible extensions to Proposition~\ref{prop:corank_2_SAP}. These are based on the complete bipartite graphs $K_{a,b}$; it is well-known and easy to check that these graphs have $-\sqrt{ab}$ and $\sqrt{ab}$ as eigenvalues (of their adjacency matrix) of multiplicity 1 and all the remaining eigenvalues are $0$. As an auxiliary tool, we show that the (negative of) adjacency matrices of these graphs can be perturbed to obtain discrete Schrödinger operators with a prescribed corank.
\begin{obs}\label{obs:complete_bipartite}
Let $G$ be a complete bipartite graph with parts $A,B$ of sizes $a, b$ such that $1\leq a, b$. Assume that $S\subseteq V(G)$ is such that $A\cap S\neq A$ and $B\cap S\neq B$. Then for all $\varepsilon>0$ small enough the matrix $M:=\varepsilon I_S-A_G$ belongs to $\mc{M}(G)$. Moreover, $\corank(M)=a+b-2-\abs{S}$ and
	\begin{equation}\label{eq:complete_bipartite_ker}
		\ker(M)=\set{x\in\R^{V(G)}\colon x^T\mbbm{1}_{A\setminus S}=0=x^T\mbbm{1}_{B\setminus S},\ \supp(x)\cap S=\emptyset}.
	\end{equation}
\end{obs}
\begin{proof}
	Given a matrix $T\in\R^{n\times n}$, we write $\lambda_k(T)$ for its $k$-smallest eigenvalue (counting multiplicities). From the min-max characterization of eigenvalues it directly follows that for every $k\in[n]$
	\begin{equation}\label{eq:Weyl}
		\lambda_k(T_1+T_2)\geq\lambda_k(T_2)
	\end{equation}
	whenever $T_1, T_2\in\R^{n\times n}$ are symmetric and $T_1$ is positive semi-definite.
	
	Besides the spectrum of $G$ presented above, one can directly see that $\ker(A_G)=\set{x\in\R^{V(G)}\colon x^T\mbbm{1}_A=0=x^T\mbbm{1}_B}$.
	If $S=\emptyset$, there is nothing left to prove. Thus, from now on we assume that $S\neq\emptyset$. Moreover, due to a complete symmetry between $A$ and $B$, we can and will assume that $B\cap S\neq\emptyset$.
	
	The eigenvalues of any matrix depend continuously on the entries of the matrix. Thus, every small enough perturbation of the entries of $-A_G$ preserves the sign of its nonzero eigenvalues. Since $M$ is a small enough perturbation of $-A_G$ by assumption, it still has at least one negative eigenvalue and at least one positive eigenvalue (corresponding to the eigenvalues $\pm\sqrt{ab}$ of $-A_G$). Additionally, as $\varepsilon I_S$ is positive semi-definite, the inequality \eqref{eq:Weyl} for $\lambda_2(M)$ says that
	\[
	\lambda_2(M)\geq\lambda_2(-A_G)=0.
	\]
	Thus, $M$ has exactly one negative eigenvalue, which implies that $M\in\mc{M}(G)$.

	Moreover, we immediately see that every $x\in\R^{V(G)}$ such that $x^T\mbbm{1}_{A\setminus S}=0=x^T\mbbm{1}_{B\setminus S}$ and $\supp(x)\cap S=\emptyset$ satisfies $Mx=\mb{o}$; these vectors span a subspace of dimension exactly $a+b-2-\abs{S}$, since $A,B\nsubseteq S$ by assumption. Therefore, $\corank(M)\geq a+b-2-\abs{S}$ and $\ker(M)$ contains the space on the right-hand side of \eqref{eq:complete_bipartite_ker}.
	
	Next, if $\supp(x)\subseteq S$ and $x^T\mbbm{1}_{A}=0=x^T\mbbm{1}_{B}$, then $Mx=\varepsilon x$. Consequently, the eigenvalue $\varepsilon>0$ has multiplicity at least $\max\set{0, \abs{A\cap S}-1}+\abs{B\cap S}-1\geq\abs{S}-2$.
	
	The above discussion establishes the sign of all but one or two eigenvalues of $M$ (depending on whether $A\cap S=\emptyset$ or not). By the spectral theorem, we can choose an eigenvector $z\neq\mb{o}$ corresponding to (one of) the unknown eigenvalue(s), denoted by $\mu$, so that it is orthogonal to all the eigenvectors of $M$ described above. In other words, we can assume that $z$ is constant on each of the sets $A\setminus S, A\cap S, B\cap S$, and $B\setminus S$. This implies that there are $\alpha, \beta, \sigma_A, \sigma_B\in \R$ such that 
	\[
		z=\alpha\mbbm{1}_{A\setminus S}+\beta\mbbm{1}_{B\setminus S}+\sigma_A\mbbm{1}_{A\cap S}+\sigma_B\mbbm{1}_{B\cap S}.
	\]
	Our task is to show that $\mu\neq 0$, which will then imply that $\corank(M)=a+b-2-\abs{S}$ and that $\ker(M)$ does not contain any vectors besides those asserted by \eqref{eq:complete_bipartite_ker}.
	
	From the equation $Mz=\mu z$ we deduce that
	\begin{align}
	\mu\alpha &= -\beta\abs{B\setminus S}-\sigma_B\abs{B\cap S} \qquad  &\mu\sigma_A=\sigma_A\varepsilon-\beta\abs{B\setminus S}-\sigma_B\abs{B\cap S}\label{eq:mu_A}, \\
	\mu\beta&=-\alpha\abs{A\setminus S}-\sigma_A\abs{A\cap S}  \qquad  &\mu\sigma_B=\sigma_B\varepsilon-\alpha\abs{A\setminus S}-\sigma_A\abs{A\cap S},\label{eq:mu_B}
	\end{align}
	with the provision that if $A\cap S=\emptyset$, then we discard the second equation in \eqref{eq:mu_A}. We recall that $\abs{A\setminus S}, \abs{B\setminus S}, \abs{B\cap S}>0$ by assumption.
	Subtracting the first equation from the second in each of \eqref{eq:mu_A} and \eqref{eq:mu_B}, we get that 
	\begin{equation}\label{eq:mu_resol}
	\mu(\sigma_A-\alpha)=\sigma_A\varepsilon \qquad\text{and}\qquad \mu(\sigma_B-\beta)=\sigma_B\varepsilon.
	\end{equation}
	
	If $\beta=\sigma_B$, then $\sigma_B=0=\beta$. In the case that $A\cap S\neq\emptyset$, we also get $\mu\sigma_A=\sigma_A\varepsilon$ from \eqref{eq:mu_A}. In any case, if $\mu=0$, we get that $z=\alpha\mbbm{1}_{A\setminus S}$, but then $Mz=-\alpha\abs{A\setminus S}\mbbm{1}_B$. Thus, either $z=\mb{o}$, or $z$ is not an eigenvector, which is a contradiction. Consequently, $\mu\neq 0$ in this case.
	
	Next, we assume $\beta-\sigma_B\neq 0$ and we deduce from \eqref{eq:mu_resol} that $\mu=\frac{\sigma_B}{\sigma_B-\beta}\varepsilon$. Hence, we see that $\mu=0$ if and only if $\sigma_B=0$. By the second equation in \eqref{eq:mu_B}, $\sigma_B=0$ implies $\alpha\abs{A\setminus S}=-\sigma_A\abs{A\cap S}$. Now if $A\cap S=\emptyset$, we actually get that $\alpha=0$, which in turn implies $\beta=0$ via the first part of \eqref{eq:mu_A}; this is a contradiction to $z\neq \mb{o}$.
	
	So we are left with the case that $\beta-\sigma_B\neq 0$ and that $A\cap S\neq\emptyset$; due to the latter, we can follow the exact same reasoning as for $\beta$ and $\sigma_B$ above and assume that $\alpha-\sigma_A\neq 0$ as well (as the other case leads to $\mu\neq 0$). Therefore, $\mu=\frac{\sigma_A}{\sigma_A-\alpha}\varepsilon$ by \eqref{eq:mu_resol}. Hence, $\mu=0$ if and only if $\sigma_A=0$. Consequently, if $\mu=0$ in the present case, we also see that $\sigma_B=\sigma_A=0$. However, this implies $\alpha=0=\beta$ by the second equations in \eqref{eq:mu_A} and \eqref{eq:mu_B}, contradicting that $z\neq\mb{o}$.
	
	In any case, we obtained that $\mu\neq 0$, which finishes the proof.
\end{proof} 

We return to the discussion of the limits on possible extensions of Proposition~\ref{prop:corank_2_SAP}. First, for $K_{1, t}$ with $t\geq 4$ and $\abs{S}=t-4$ we obtain from Observation~\ref{obs:complete_bipartite} a discrete Schrödinger operator $M_t$ of corank 3. On the other hand, $K_{1, t}$ is an outer-planar graph, so $\mu(K_{1, t})\leq 2$ by \citep[Thm.~5.7]{CdV_orig_en}, and thus, $M_t$ does not satisfy (SAP) for any $t\geq 4$. This shows that Proposition~\ref{prop:corank_2_SAP} does not generalize to corank 3 matrices without tightening the assumptions.

Second, as was mentioned in the introduction, we observe that for every $k\in\N, k\geq 3$, there are discrete Schrödinger operators on $k$-connected graphs that have corank as low as $4$, yet they do not have (SAP).

\begin{obs}\label{obs:quest_on_SAP_tight}
	Let $a,b\in\N, a\geq 3, b\geq 4$, and $a\leq b$. Then there is $M\in\mc{M}(K_{a,b})$ with $\corank(M)=4$ that does not satisfy (SAP).
\end{obs}
\begin{proof}
	We write $A,B$ for the parts of $K_{a,b}$ of sizes $a,b$, respectively. We choose $S\subset A\cup B$ such that $\abs{A\cap S}=a-2$ and $\abs{B\cap S}=b-4$. Then Observation~\ref{obs:complete_bipartite} provides $M\in\mc{M}(K_{a,b})$ with $\corank(M)=a+b-2-(a+b-6)=4$.
	
	Let $v_1,\ldots, v_4\in B\setminus S$ be pairwise distinct. By \eqref{eq:complete_bipartite_ker}, both $x:=\mbbm{1}_{\set{v_1}}-\mbbm{1}_{\set{v_2}}$ and $y:=\mbbm{1}_{\set{v_3}}-\mbbm{1}_{\set{v_4}}$ belong to $\ker(M)$. We set $X:=xy^T+yx^T$. Then $MX=\mb{0}$ and $X$ also satisfies the other conditions in Definition~\ref{def:SAP} certifying that $M$ does not have (SAP).
\end{proof}

As the last remark, we note that Observation~\ref{obs:quest_on_SAP_tight} does not provide counterexamples to Question~\ref{q:SAP_and_connectivity}, since $\mu(K_{a,b})=a+1$ whenever $a\leq b$ and $b\geq 3$ \citep[Eq.~(2)]{CdV_main}, whereas $K_{a,b}$ is only $a$-connected. It shows that the assumption of $\mu(G)$-connectedness in Question~\ref{q:SAP_and_connectivity} cannot be relaxed to $(\mu(G)-1)$-connectedness for any value of $\mu(G)\geq 4$, but this was already noted by \citet{SS_flat_and_SAP}.
\section{Graphs on surfaces}\label{s:surfaces}
We begin with several preliminaries specific to the present section.

\paragraph{Surfaces, graph embeddings.}
By a \emph{surface} we always mean a connected closed surface.
A \emph{face} of an embedding $\phi\colon G\to S$ of a graph $G$ into a surface $S$ is a connected component of $S\setminus\phi(G)$. A \emph{map} of a connected graph $G$ on a surface $S$ is an embedding of $G$ into $S$ such that every open face is homeomorphic to an open disk; it is a \emph{triangulation} of $S$ if every face is additionally bounded by exactly three edges. Writing $v, e, f$ for the number of vertices, edges and faces, respectively, of a map $\phi\colon G\to S$, the classical Euler's formula asserts that
\begin{equation}\label{eq:Euler}
	v-e+f=\chi(S),
\end{equation}
where $\chi(S)$ is an invariant of $S$ called \emph{Euler's characteristic}. 

When $S$ is \emph{orientable}, any map on $S$ can be described purely combinatorially, up to a homeomorphism of $S$, by its \emph{rotation system} (sometimes called a \emph{combinatorial map})---a cyclic permutation at each vertex describing the relative ordering of the incident edges. For more details as well as general background on surfaces and graph embeddings we refer to \citep{Mohar_Thomassen_GraphsOnSurface}.

According to the classification of surfaces (e.g., \citep{Mohar_Thomassen_GraphsOnSurface}), every \emph{orientable} surface is homeomorphic to the sphere with $n\in\N$ handles attached to it; we denote such surface by $T_n$. The number of handles $n$ is referred to as the (orientable) \emph{genus} of $T_n$. The \emph{genus of a map} on an orientable surface $S$ is the genus of $S$ and equals to $(2-\chi(S))/2$. By \eqref{eq:Euler}, the genus of a map is determined by the number of its vertices, edges and faces.

\paragraph{Automorphisms, rotary maps.}
A triple incidence consisting of a vertex, an edge and a face of a map (i.e., one of two sides of a directed edge of a map) is called a \emph{blade}. A permutation of the blades preserving all incidences between vertices, edges and faces is called an \emph{automorphism} of the map; these form a group under composition, which is called the \emph{automorphism group} of the map. Since the underlying graph of a map is connected, if we know the image of a single blade under an automorphism, we can uniquely reconstruct the whole automorphism.

A particular highly symmetric type of maps comprises of so-called \emph{rotary} (also called \emph{regular}) maps; these are maps for which there are two automorphisms $y, z$ such that the induced action of $y$ on the edges of the underlying graph cyclically permutes the successive edges of some face $\sigma$ and the induced action of $z$ cyclically permutes the successive edges incident to a vertex of $\sigma$. The connectivity of the graph then implies that the automorphism group acts transitively on the vertices, the edges and also the faces of the map (but not necessarily on its blades). In particular, every vertex has then the same degree $q$ and every face is incident to $p$ edges for some $p, q\in\N$. Euler's formula \eqref{eq:Euler} provides a linear relation between $p, q$ and the genus of such a map; the tuple $(p, q)$, called a \emph{type} of a rotary map, thus determines its genus. For more details, we refer to \citep{Conder--Dobcsanyi, Bergau--Garbe} and \citep[Ch.~8]{Coxeter_Moser}.

\subsection*{The counterexample to Conjecture~\ref{c:CdV_surfaces}}
Given the tight relationship between a rotary map and its automorphism group, starting with an appropriate presentation of the group, one can reconstruct the map. However, the particular way of doing it depends on further properties of the group; below we recall only the construction that applies to our case of interest. For the discussion of the other cases we refer to \citep{Conder--Dobcsanyi}.

The counterexample to Conjecture~\ref{c:CdV_surfaces} described here (and which was also used by \citep{CdV_surfaces_cont_counter-examples}) comes from a particular rotary map of genus $10$. \citet{Conder--Dobcsanyi} enumerated all rotary orientable maps of genus up to $15$. Later on, \citet{Conder_list} enumerated all orientable rotary maps of genus up to $101$. The group we are interested in is called \verb|C10.1| in \citep{Conder--Dobcsanyi} (and \verb|T10.1| in \citep{Conder_list} in a slightly different, but equivalent form) and is presented as\footnote{The presentation in \citep{Conder--Dobcsanyi} contains apparently a typo, as it is missing the relation $(yz)^2=1$. However, this relation holds true in every automorphism group of any rotary map. The corresponding relation is also present in the presentation of the same group in \citep{Conder_list}.}
\begin{equation}\label{eq:Gamma}
\Gamma_{10}:=\langle y, z\,|\,1=y^3=z^8=(yz)^2=z^2y^{-1}z^3y^{-1}zy^{-1}z^{-3}yz^{-3}y^{-1}\rangle.	
\end{equation}

	Following \citep{Conder--Dobcsanyi}, the map it corresponds to\footnote{In reference to \citep{Conder--Dobcsanyi}, it is important to note that the map underlying the group $\Gamma_{10}$ is of so-called \emph{chiral} type. This influences the exact way it is reconstructed from its automorphism group.}, here denoted as $\phi_{10}$, has as vertices the right cosets of $\langle z\rangle$ in $\Gamma_{10}$, the edges are the right cosets of $\langle yz\rangle$ and the faces are the right cosets of $\langle y\rangle$. Incidences are determined by non-empty intersection. The graph underlying $\phi_{10}$ is denoted by $G_{10}$. To create $G_{10}$ in \verb|SageMath|, the reader can run the method \verb|construct_G10()| from the ancillary file \verb|code.sage|.
	
\begin{prop}\label{prop:counter_ex}
	The map $\phi_{10}$ is a triangulation of $T_{10}$ such that $\mu(G_{10})\geq 16$.
\end{prop}
\begin{proof}[Proof using \texttt{SageMath}~\citep{sagemath}] All the code used in the present proof is contained in the ancillary file \verb|code.sage|. In order to reproduce all the steps of the proof, the reader can run the method \verb|genus10()|.
	
	By \citep{Conder--Dobcsanyi} or using \verb|SageMath| (also a part of the method \verb|construct_G10()|), we see that $\abs{\Gamma_{10}}=432$ and the orders of $y, z$, and $yz$ are $3, 8$, and $2$, respectively. Since $\Gamma_{10}$ acts transitively on $\phi_{10}$, it has $432/8=54$ vertices (each of degree 8), $432/2=216$ edges, and $432/3=144$ faces, each of which is a triangle. Using Euler's formula \eqref{eq:Euler}, the genus of $\phi_{10}$ is $10$.
	
	Using \verb|SageMath| again (via the method \verb|factor_charpoly()|), the characteristic polynomial of $G_{10}$ is
	\[p_{G_{10}}(x)=(x - 8)(x + 4)^2 x^3 (x + 1)^8 (x + 3)^8 (x^2 - 2x - 6)^{16},\]
	which gives the spectrum $\set{-4^{(2)}, -3^{(8)}, 1-\sqrt{7}^{(16)}, -1^{(8)}, 0^{(3)}, 1+\sqrt{7}^{(16)}, 8^{(1)}}$.
	Thus, setting $\lambda:=1+\sqrt{7}$ and defining a matrix $M_{10}\in\mc{M}(G_{10})$ as
	\[
	M_{10}:=\lambda I_{G_{10}}-A_{G_{10}},
	\]
	we obtain a discrete Schrödinger operator with kernel of dimension $16$.
	
	It remains to check that $M_{10}$ satisfies (SAP). In general, for a graph $G$, Definition~\ref{def:SAP} yields a homogeneous system of $\abs{V(G)}^2$ linear equations with $\binom{\abs{V(G)}}{2}-\abs{E(G)}$ variables. We can omit $\abs{V(G)}$ equations by noting that the equations corresponding to the diagonal elements of the right-hand side in Definition~\ref{def:SAP} are always automatically satisfied:
	\[\sum_{v\in V(G)} M_{uv}X_{vu}=M_{uu}X_{uu}+\sum_{v\in V(G):uv\in E(G)}M_{uv}X_{vu}=0\quad\text{for every $u\in V(G)$}.\]
	
	For $G_{10}$ we get $54^2-54=2862$ equations and $1215$ variables. We write $T_{10}$ for the corresponding $2862\times 1215$ matrix describing this homogeneous linear system. By the above, $M_{10}$ has (SAP) if and only if $T_{10}$ has full rank. Since $\lambda$ is irrational and we want to perform the computation in \emph{exact} arithmetic, we use that the elements of $M_{10}$, and consequently, of $T_{10}$ lie in the algebraic number field $\Q[\lambda]$ by definition.
	
	\verb|SageMath| verified that $T_{10}$ has full rank in roughly 1 hour on a conventional laptop; the corresponding code is in the method \verb|testSAP()|. We used \verb|SageMath 9.8|~\cite{sagemath} with \verb|Python 3.10.12|. Internally, our version of \verb|SageMath| uses \verb|GAP 4.11.1| for group computations as well as the libraries \verb|linbox 1.6.3.p1| and \verb|Pari 2.15.2.p1| for algebraic computations.
\end{proof}

Every graph $G$ that embeds into a surface $S$ embeds also to any other surfaces $S'$ with $\chi(S')<\chi(S)$ \citep[Prop.~4.4.1]{Mohar_Thomassen_GraphsOnSurface}.
Thus, comparing the bound of Proposition~\ref{prop:counter_ex} with \eqref{eq:Heawood}, we conclude that $G_{10}$ is a counterexample to Conjecture~\ref{c:CdV_surfaces} also for all surfaces $S$ with $2-2\cdot 10=-18>\chi(S)\geq -28$ (including the non-orientable surfaces in the given range).
\begin{prop}\label{prop:counter_ex_extended}
	Conjecture~\ref{c:CdV_surfaces} fails for all surfaces $S$ with $\chi(S)\in [-28, -18)$.
\end{prop}

Moreover, by Theorem~\ref{thm:discrete_to_cont} applied to $G_{10}$ we obtain Schrödinger operators
providing counterexamples to the continuous Conjecture~\ref{c:CdV_surfaces_continuous} in the specified range for $\chi(S)$. As far as we know, this is a new result; in particular, it does not follow from the counterexamples found in \citep{CdV_surfaces_cont_counter-examples}, which inspired this note. On the other hand, counterexamples obtained via Theorem~\ref{thm:discrete_to_cont} are in general only Schrödinger operators, whereas \citep{CdV_surfaces_cont_counter-examples} obtained Riemannian Laplacians.
\section{Discussion and open problems}
\paragraph{Strong Arnold Property.}
It is the authors' view that the current lack of understanding of the relationship between (SAP) for $M\in\mc{M}(G)$ on one side and the combinatorics of $G$ and $\ker(M)$ on the other side is a major obstacle for further progress on various interesting conjectures regarding $\mu$.
From this perspective, we see Question~\ref{q:SAP_and_connectivity} as an important study case, which might generate a deeper insight into combinatorial ramifications of (SAP). A necessary condition for the validity of Question~\ref{q:SAP_and_connectivity}, which has also been mentioned in \citep{SS_flat_and_SAP}, is the following: Is it true that $\max_{M\in\mc{M}(G)}\corank(M)=\mu(G)$ for every $\mu(G)$-connected graph $G$?

The weaker version of Question~\ref{q:SAP_and_connectivity} has been studied earlier than Question~\ref{q:SAP_and_connectivity} and its validity in the cases $\mu(G)\leq 4$ was first established for the weaker version only \citep{vdHolst_planar, CdV_linkless}. Both versions are currently open for all values of $\mu(G)\geq 5$.

\paragraph{The parameter $\mu$ and graphs on surfaces.}
It is natural to ask whether one can prove Proposition~\ref{prop:counter_ex} without resorting to any kind of computer-aided verification. There are two principal steps in this direction. The first one is to establish that the matrix $M_{10}$ constructed in the proof of Proposition~\ref{prop:counter_ex} has exactly one negative eigenvalue and that its kernel has dimension at least $16$. This problem might be approachable through the means of group representation theory applied to the group $\Gamma_{10}$, perhaps along the lines followed by \citep{CdV_surfaces_cont_counter-examples}.

The second problem is to establish that $M_{10}$ has (SAP). Currently, there are basically no tools to do so. All the results in the literature that establish (SAP) for some matrix either use ad hoc arguments applicable only to a narrow class of matrices in question, or bypass the problem somehow, for instance, using the results presented in the introduction around Question~\ref{q:SAP_and_connectivity} and Proposition~\ref{prop:corank_2_SAP}. In this connection, we note that $M_{10}$ is very far from the scope of Question~\ref{q:SAP_and_connectivity}, as it is $8$-connected, but $\corank(G_{10})\geq 16$. We showed in Observation~\ref{obs:quest_on_SAP_tight} in Section~\ref{s:SAP} that one cannot hope to get (SAP) for free for discrete Schrödinger operators on graphs $G$ that are not $\mu(G)$-connected. Thus, even if one manages to establish that $\corank(M_{10})\geq 16$ by hand, we currently do not see any way to prove that $M_{10}$ has (SAP) without using a computer.

Since the answer to Conjecture~\ref{c:CdV_surfaces} is negative, we wonder what is the correct order of growth of $\max\set{\mu(G)\colon G\text{ embeds into } S}$ as $\chi(S)\to-\infty$. In particular, could Conjecture~\ref{c:CdV_surfaces} be still true up to lower order terms?  

\subsection*{Acknowledgements.}
We would like to express our sincere gratitude to the developers and maintainers of \verb|SageMath| and of the systems it depends on for the computation conducted in the present paper, such as \verb|Python|, \verb|GAP|, \verb|linbox| and \verb|Pari| and the packages on which these depend.

\printbibliography

\end{document}